\documentclass{amsart}
\usepackage{amsfonts, amsbsy, amsmath, amssymb}

\newtheorem{thm}{Theorem}[section]
\newtheorem{lem}[thm]{Lemma}

\newtheorem{prop}[thm]{Proposition}

\newtheorem{rmk}[thm]{Remark}

\newtheorem{thm-con}[thm]{Theorem-Conjecture}
\numberwithin{equation}{section}

\theoremstyle{definition}

\newcommand{\f}{\Bbb F}

\begin{document}

\title[Reversed Dickson Polynomials of the Third Kind]{Reversed Dickson polynomials of the third kind}

\author[Neranga Fernando]{Neranga Fernando}
\address{Department of Mathematics,
Northeastern University, Boston, MA 02115}
\email{w.fernando@neu.edu}

\begin{abstract}
Let $p$ be a prime and $q=p^e$. We discuss the properties of the reversed Dickson polynomial $D_{n,2}(1,x)$ of the third kind. We also give several necessary conditions for the reversed Dickson polynomial of the third kind $D_{n,2}(1,x)$ to be a permutation of $\f_{q}$. In particular, we give explicit evaluation of the sum $\sum_{a\in \f_q}D_{n,2}(1,a)$. 
\end{abstract}



\maketitle

%%%%%%%%%%%%%%%%%%%%%%%%%%%%%%%%%%%%%%%%
%  section 1 
%%%%%%%%%%%%%%%%%%%%%%%%%%%%%%%%%%%%%%%%
\section{Introduction}
Let $p$ be a prime and $q$ a power of $p$. Let $\Bbb F_q$ be the finite field with $q$ elements. A polynomial $f \in \Bbb F_q[{\tt x}]$ is called a \textit{permutation polynomial} (PP) of $\Bbb F_q$ if the mapping $x\mapsto f(x)$ is a permutation of $\Bbb F_q$. In the study of permutation polynomials over finite fields, Dickson polynomials have played a pivotal role. 

The $n$-th Dickson polynomial of the first kind $D_n(x,a)$ is defined by
\[
D_{n}(x,a) = \sum_{i=0}^{\lfloor\frac n2\rfloor}\frac{n}{n-i}\dbinom{n-i}{i}(-a)^{i}x^{n-2i},
\]
where $a\in \f_q$ is a parameter. 

The permutation property of the Dickson polynomials of the first kind is completely known. When $a=0$, $D_{n}(x,a) = x^n$, which is a PP over $\Bbb F_q$ if and only if $(n,q-1)=1$. When $0\neq a \in \mathbb{F}_{q}$, $D_{n}(x,a)$ is a PP over $\Bbb F_q$ if and only if $(n,q^2-1)=1$; see \cite[Theorem ~7.16]{Lidl-Niederreiter-97} or \cite[Theorem ~3.2]{Lidl-Mullen-Turnwald-1993}.

The $n$-th Dickson polynomial of the second kind $E_n(x,a)$ is defined by
\[
E_{n}(x,a) = \sum_{i=0}^{\lfloor\frac n2\rfloor}\dbinom{n-i}{i}(-a)^{i}x^{n-2i},
\]
where $a\in \f_q$ is a parameter. 

The permutation behavior of the Dickson polynomials of the second kind  has been extensively studied by many authors. We refer the reader to \cite{Cohen-1994} for more details about the Dickson polynomials of the second kind.

Dickson polynomials are closely related to the well-known Chebyshev polynomials over the complex numbers by
$$D_n(2x,1)=2T_n(x)\hspace{1cm}\textnormal{and}\hspace{1cm}E_n(2x,1)=U_n(x),$$

where $T_n(x)$ and $U_n(x)$ are Chebyshev polynonmials of degree $n$ of the first kind and the second kind, respectively. 

The $n$-th reversed Dickson polynomial of the first kind $D_{n}(a,x)$ was first introduced by Hou, Mullen, Sellers and Yucas in \cite{Hou-Mullen-Sellers-Yucas-FFA-2009} by reversing the roles of the variable and the parameter in the $n$-th Dickson polynomial of the first kind $D_{n}(x,a)$. It was shown that when $a=0$, $D_{n}(0,x)$ is a PP over $\Bbb F_q$ if and only if $n = 2k$ with $(k,q-1)=1$. Also, when $a\neq0$, 
\[
D_{n}(a,x) = a^n D_{n}(1,\frac{x}{a^2}).
\]
Hence $D_{n}(a,x)$ is a PP on $\mathbb{F}_{q}$ if and only if $D_{n}(1,x)$ is a PP on $\mathbb{F}_{q}$. 

In \cite{Hou-Ly-FFA-2010}, Hou and Ly further studied the reversed Dickson polynomials of the first kind $D_{n}(1,x)$ and explained several necessary conditions for $D_{n}(1,x)$ to be a permutation of $\f_q$. 

Recently, Hong, Qin, and Zhao studied reversed Dickson polynomials of the second kind in \cite{Hong-Qin-Zhao-FFA-2016}. They presented several necessary conditions for the reversed Dickson polynomial of the second kind $E_n(1,x)$ to be a permutation of $\f_q$. 

In \cite{Wang-Yucas-FFA-2012}, Wang and Yucas introduced the $n$-th Dickson polynomial of the $(k+1)$-th kind and the $n$-th reversed Dickson polynomial of the $(k+1)$-th kind.

For $a\in \f_q$, the $n$-th Dickson polynomial of the $(k+1)$-th kind $D_{n,k}(x,a)$ is defined by 
\[
D_{n,k}(x,a) = \sum_{i=0}^{\lfloor\frac n2\rfloor}\frac{n-ki}{n-i}\dbinom{n-i}{i}(-a)^{i}x^{n-2i}.
\]

For $a\in \f_q$, the $n$-th reversed Dickson polynomial of the $(k+1)$-th kind $D_{n,k}(a,x)$ is defined by 
\begin{equation}\label{E1.1}
D_{n,k}(a,x) = \sum_{i=0}^{\lfloor\frac n2\rfloor}\frac{n-ki}{n-i}\dbinom{n-i}{i}(-x)^{i}a^{n-2i}.
\end{equation}

Then clearly $D_{n,0}(a,x)=D_{n}(a,x)$ and $D_{n,1}(a,x)=E_{n}(a,x)$. In \cite{Wang-Yucas-FFA-2012}, they defined

\begin{equation}\label{E1.2}
D_{0,k}(a,x)=2-k.
\end{equation}

They also showed that $D_{n,k}(x,a)=kE_{n}(x,a)-(k-1)D_{n}(x,a)$. A simple computation shows that the reversed Dickson polynomials agree with the above equation as well, i.e. 
\begin{equation}\label{E1.3}
D_{n,k}(a,x)=kE_{n}(a,x)-(k-1)D_{n}(a,x).
\end{equation}

In \cite{Wang-Yucas-FFA-2012}, Wang and Yucas completely described the permutation behavior of the Dickson polynomials of the third kind $D_{n,2}(x,a)$ over any prime field, but the permutation property of $D_{n,2}(x,a)$ over an arbitrary finite field is still an open problem. 

The purpose of the present paper is to explore the permutation behavior of the reversed Dickson polynomials of the third kind. By \eqref{E1.1}, the $n$-th reversed Dickson polynomial of the third kind $D_{n,2}(a,x)$ is given by 

\begin{equation}\label{E1.4}
D_{n,2}(a,x) = \sum_{i=0}^{\lfloor\frac n2\rfloor}\frac{n-2i}{n-i}\dbinom{n-i}{i}(-x)^{i}a^{n-2i}.
\end{equation}

Throughout the paper, we denote the $n$-th reversed Dickson polynomial of the third kind $D_{n,2}(a,x)$ by $F_{n}(a,x)$. Here is an overview of the paper. 

In Section 2, we present several properties of the reversed Dickson polynomials of the third kind. In section 3, we survey some miscellaneous necessary conditions for $F_{n}(a,x)$ to be a permutation of $\f_q$. In section 4, we compute the sum  $\sum_{a\in \f_q}F_n(1,a)$.

%%%%%%%%%%%%%%%%%%%%%%%%%%%%%%%%%%%%
%   section 2
%%%%%%%%%%%%%%%%%%%%%%%%%%%%%%%%%%%%

\section{Reversed Dickson polynomials of the third kind}
We study the properties of reversed Dickson polynomials of the third kind in this section. 

\begin{lem}\label{L2.1} $F_n (a,x)$ is not a PP when $a=0$.
\end{lem}

\begin{proof}

When $a=0$, the reversed Dickson polynomials of the first kind satisfy (See \cite{Hou-Mullen-Sellers-Yucas-FFA-2009})
\[
D_n (0,x)=
\begin{cases}
0&\text{if}\ n\,\, \text{is odd},\cr
2\,(-x)^k&\text{if}\ n = 2k,
\end{cases}
\]
and the reversed Dickson polynomials of the second kind satisfy (See \cite{Hong-Qin-Zhao-FFA-2016})
\[
E_n (0,x)=
\begin{cases}
0&\text{if}\ n\,\, \text{is odd},\cr
(-x)^k&\text{if}\ n = 2k.
\end{cases}
\]
From \eqref{E1.3}, we have $F_n(0,x)=2E_n(0,x)-D_n(0,x)$ which implies $F_n (0,x)=0$ for all $n$. Hence $F_n (a,x)$ is not a PP when $a=0$.

\end{proof}

We thus hereafter assume that $a\in \f_{q}^{*}$. 

\begin{lem}\label{L2.2}
For $a\ne 0$, Let $x=y+ay^{-1}$ for some $y\in \f_{q^2}$ with $y\neq 0$ and $y^2\neq a$. Then the functional equation of $F_n(a,x)$ is given by 
$$F_n(a,x)=\displaystyle\frac{a}{2y-a}(y^n-(a-y)^n),\, \textnormal{where}\,\, y\neq \frac{a}{2}.$$
\end{lem}

\begin{proof}
Note that

\[
D_{n}(a,x)=y^n+(a-y)^n
\]

and 

\[
E_n(a,x)=\displaystyle\frac{y^{n+1}-(a-y)^{n+1}}{2y-a}
\]

are the functional expressions of the Dickson polynomial of the first kind and second kind, respectively. 
Hence the rest of the proof immediately follows from \eqref{E1.3}. 
\end{proof}

Let  $a\in \f_{q}^{*}$. Then it follows from \eqref{E1.4} that
\begin{equation}\label{E2.1}
F_{n}(a,x)=a^n\,F_{n}(1,\frac{x}{a^2}). 
\end{equation}
Hence $F_{n}(a,x)$ is a PP on $\f_{q}$ if and only if $F_{n}(1,x)$ is a PP on $\f_{q}$.

\begin{thm}\label{T2.2}
let $p$ be an odd prime, $n$ and $k$ be positive integers. Then we have the following. 
\begin{itemize}
\item [(1)] If $y\neq \frac{1}{2}$, then $F_n(1,y(1-y))=\dfrac{y^n-(1-y)^n}{2y-1}$. Also, $F_{n}(1,\frac{1}{4})=\frac{n}{2^{n-1}}$. 
\item [(2)] If $\text{gcd}(n,k)=1$, then $F_{np^k}(1,x)=(F_n(1,x))^{p^k}(1-4x)^{\frac{p^k-1}{2}}$.
\item [(3)] If $n_1\equiv n_2 \pmod{q^2-1}$, then $F_{n_1}(1,x_0)=F_{n_2}(1,x_0)$ for any $x_0\in \f_{q}\setminus \{\frac{1}{4}\}$.
\end{itemize}
\end{thm}

\begin{proof}

\begin{itemize}
\item [(1)] Let $a=1$ in Lemma~\ref{L2.2}, and write $x=y(1-y)$. Then for $y\neq \frac{1}{2}$, we have 
$$F_n(1,y(1-y))=\dfrac{y^n-(1-y)^n}{2y-1}.$$
When $a=2$ and $x=1$, from \eqref{E2.1} we have  
$$F_{n}(2,1)=2^n\,F_{n}(1,\frac{1}{4}).$$
which implies
$$F_{n}(1,\frac{1}{4})=\displaystyle\frac{F_{n}(2,1)}{2^n}.$$
We have $D_n(2,1)=2$ and $E_n(2,1)=n+1$ (See \cite{Lidl-Mullen-Turnwald-1993}). Then it follows from \eqref{E1.3} that 
$$F_n(2,1)=2E_n(2,1)-D_n(2,1)=2(n+1)-2=2n.$$
Hence
$$F_{n}(1,\frac{1}{4})=\displaystyle\frac{n}{2^{n-1}}.$$

\item [(2)] Let $x=y(1-y)$ with $y\neq \frac{1}{2}$. Then we have
\[
\begin{split}
&F_{np^k}(1,x)=F_{np^k}(1,y(1-y))= \displaystyle\frac{y^{np^k}-(1-y)^{np^k}}{2y-1}\cr
&=\displaystyle\frac{(y^{n}-(1-y)^{n})^{p^k}}{2y-1}\cr
&=\displaystyle\frac{(y^{n}-(1-y)^{n})^{p^k}}{(2y-1)^{p^k}}\,\,(2y-1)^{p^k-1}\cr
&=(\displaystyle\frac{y^{n}-(1-y)^{n}}{2y-1} )^{p^k}\,\,(2y-1)^{p^k-1}\cr
&=(F_n(1,y(1-y)))^{p^k}\,\,(2y-1)^{p^k-1}\cr
&=F_n(1,x)^{p^k}\,\,(2y-1)^{p^k-1} \cr
&=F_n(1,x)^{p^k}\,\,(1-4x)^{\frac{p^k-1}{2}}.
\end{split}
\]
 If $y=\frac{1}{2}$, then 
$$F_{np^k}(1,\frac{1}{4})= \frac{np^k}{2^{np^k}-1}=0=F_{n}(1,x)^{p^k}(1-4x)^{\frac{p^k-1}{2}}.$$

\item [(3)]  For $x_0\in \f_{q}\setminus \{\frac{1}{4}\}$, there exists $y_0\in \f_{q^2}\setminus \{\frac{1}{2}\}$ such that $x_0=y_0(1-y_0)$.
Then we have
\[
\begin{split}
F_{n_1}(1,x_0)&=\dfrac{y_0^{n_1}-(1-y_0)^{n_1}}{2y_0-1}\cr
&=\dfrac{y_0^{n_2}-(1-y_0)^{n_2}}{2y_0-1}\cr
&=F_{n_2}(1,x_0).
\end{split}
\]

\end{itemize}

\end{proof}

\begin{rmk}
If $char(\f_q)=2$, then $F_n(1,x)$ is the $n$-th reversed Dickson polynomial of the first kind $D_{n}(1,x)$ since from (1) in Theorem~\ref{T2.2} we have
$$F_n(1,x(1-x))=x^n+(1-x)^n=D_{n}(1,x(1-x)).$$
We thus hereafter always assume, unless specified, in this paper that $p$ is odd. 
\end{rmk}

\begin{prop}Let $p$ be an odd prime and $n$ be a non-negative integer. Then 
$$F_0(1,x)=0,\,\,F_1(1,x)=1,\,\,\textnormal{and}$$
$$F_n(1,x)=F_{n-1}(1,x)-x\,F_{n-2}(1,x),\,\, \textnormal{for} \,\,n\geq 2.$$
\end{prop}
\begin{proof}
It follows from Theorem~\ref{T2.2} (1) that $F_0(1,x)=0,\,\,F_1(1,x)=1$. 

Let $n\geq 2$. When $x\neq \frac{1}{4}$, we write $x=y(1-y)$ with $y\neq \frac{1}{2}$. By Theorem~\ref{T2.2} (1), we have 
\[
\begin{split}
F_{n-1}(1,x)-x\,F_{n-2}(1,x)&= F_{n-1}(1,y(1-y))\,-\,y(1-y)\,F_{n-2}(1,y(1-y))\cr
&=\dfrac{y^{n-1}-(1-y)^{n-1}}{2y-1}\,-\,y(1-y)\,\dfrac{y^{n-2}-(1-y)^{n-2}}{2y-1}\cr
&=\dfrac{y^n-(1-y)^n}{2y-1}=F_n(1,y(1-y))=F_n(1,x).
\end{split}
\]
When $x=\frac{1}{4}$, 
$$F_{n-1}(1,\frac{1}{4})-\frac{1}{4}\,F_{n-2}(1,\frac{1}{4})=\frac{n-1}{2^{n-2}}-\frac{1}{4}\frac{n-2}{2^{n-3}}=\frac{n}{2^{n-1}}=F_n(1,\frac{1}{4}).$$
\end{proof}

\begin{thm}
Let $p$ be an odd prime. $q=p^e,\,e,k\in \mathbb{Z}^+$, $1\leq k\leq e$. Then 
$F_{p^k}(1,x)$ is a PP of $\f_{q}$ if and only if $\Big( \frac{p^k-1}{2}, q-1\Big)=1.$
\end{thm}

\begin{proof}
Let $n=1$ in Theorem~\ref{T2.2} (2). Since $F_1(1,x)=1$, we have
$F_{p^k}(1,x)=(F_1(1,x))^{p^k}(1-4x)^{\frac{p^k-1}{2}}=(1-4x)^{\frac{p^k-1}{2}}$.
Hence the proof. 
\end{proof}

\begin{thm}
Let $p$ be an odd prime. $q=p^e,\,e,k\in \mathbb{Z}^+$, $1\leq k\leq e$. Then 
$F_{2\cdot p^k}(1,x)$ is a PP of $\f_{q}$ if and only if $\Big( \frac{p^k-1}{2}, q-1\Big)=1.$
\end{thm}

\begin{proof}
Let $n=2$ in Theorem~\ref{T2.2} (2). Since $F_2(1,x)=1$, we have
$F_{2\cdot p^k}(1,x)=(F_2(1,x))^{p^k}(1-4x)^{\frac{p^k-1}{2}}=(1-4x)^{\frac{p^k-1}{2}}$. Hence the proof. 
\end{proof}

\begin{thm}\label{T2.6}
The generating function of $F_n(1,x)$ is given by 
$$\displaystyle\sum_{n=0}^{\infty}\,F_{n}(1,x)\,z^n=\displaystyle\frac{z}{1-z+xz^2}.$$
\end{thm}

\begin{proof}
\[
\begin{split}
(1-z+xz^2)\displaystyle\sum_{n=0}^{\infty}\,F_{n}(1,x)\,z^n &=\displaystyle\sum_{n=0}^{\infty}\,F_{n}(1,x)\,z^n-\,\displaystyle\sum_{n=0}^{\infty}\,F_{n}(1,x)\,z^{n+1}+\,x\,\displaystyle\sum_{n=0}^{\infty}\,F_{n}(1,x)\,z^{n+2}\cr
&=F_0(1,x)+F_1(1,x)z-F_0(1,x)z \cr &+\displaystyle\sum_{n=0}^{\infty}\,(F_{n+2}(1,x)-F_{n+1}(1,x)+xF_{n}(1,x))\,z^{n+2}
\end{split}
\]
Since $F_0(1,x)=0$, $F_1(1,x)=1$, and $F_{n+2}(1,x)=F_{n+1}(1,x)-xF_{n}(1,x)$ for $n\geq 0$, we have the desired result. 
\end{proof}

\begin{lem}\label{L2.6}(See \cite{Hou-Mullen-Sellers-Yucas-FFA-2009})
Let $q=p^e$ and Let $x\in \f_{q^2}$. Then 
$$ x(1-x)\in \f_q \,\,\textnormal{if and only if} \,\,x^q=x \,\,\textnormal{or}\,\, x^q=1-x.$$
Also, if we define
$$V=\{x\in \f_{q^2}\,;\,x^q=1-x\},$$
then $\f_q \cap V=\{\frac{1}{2}\}$. 
\end{lem}

\begin{thm}
Let $p$ be an odd prime. Then $F_n(1,x)$ is a PP of $\f_q$ if and only if the function $y \mapsto \displaystyle\frac{y^n-(1-y)^n}{2y-1}$ is a 2-to-1 mapping on $(\f_q \cup V)\setminus \{\frac{1}{2}\}$ and $\displaystyle\frac{y^n-(1-y)^n}{2y-1}\neq \frac{n}{2^{n-1}}$ for any $y\in (\f_q \cup V)\setminus \{\frac{1}{2}\}$. 
\end{thm}

\begin{proof}
For necessity, assume that $F_n(1,x)$ is a PP of $\f_q$ and $y_1, y_2\in (\f_q \cup V)\setminus \{\frac{1}{2}\}$ such that 
$\displaystyle\frac{y_1^n-(1-y_1)^n}{2y_1-1}=\displaystyle\frac{y_2^n-(1-y_2)^n}{2y_2-1}$. Then $y_1(1-y_1), y_2(1-y_2) \in \f_q$ and $F_n(1,y_1(1-y_1))=F_n(1,y_2(1-y_2))$. Since $F_n(1,x)$ is a PP of $\f_q$, we have $y_1(1-y_1)=y_2(1-y_2)$ which implies that $y_1=y_2$ or $1-y_2$. So $y \mapsto \displaystyle\frac{y^n-(1-y)^n}{2y-1}$ is a 2-to-1 mapping on $(\f_q \cup V)\setminus \{\frac{1}{2}\}$. If $y\in (\f_q \cup V)\setminus \{\frac{1}{2}\}$, then $y(1-y)\in \f_q$ and $y(1-y)\neq \frac{1}{2}(1-\frac{1}{2})$. Thus $\displaystyle\frac{y^n-(1-y)^n}{2y-1}=F_n(1,y(1-y))\neq F_n(1,\frac{1}{2}(1-\frac{1}{2}))=\frac{n}{2^{n-1}}$. 

For sufficiency, assume $x_1, x_2 \in \f_q$ such that $F_n(1,x_1)=F_n(1,x_2)$. Write $x_1=y_1(1-y_1)$ and $x_2=y_2(1-y_2)$, where $y_1, y_2 \in (\f_q \cup V)$. Then 
$$\displaystyle\frac{y_1^n-(1-y_1)^n}{2y_1-1}=F_n(1,x_1)=F_n(1,x_2)=\displaystyle\frac{y_2^n-(1-y_2)^n}{2y_2-1}.$$
If $y_1=\frac{1}{2}$, then 
$$F_n(1,x_2)=F_n(1,x_1)=F_n(1,\frac{1}{4})=\frac{n}{2^{n-1}},$$
which implies that $y_2=\frac{1}{2}$. Hence $x_1=x_2$. 

If $y_1, y_2 \neq \frac{1}{2}$, since  $y \mapsto \displaystyle\frac{y^n-(1-y)^n}{2y-1}$ is a 2-to-1 mapping on $(\f_q \cup V)\setminus \{\frac{1}{2}\}$, we have $y_1=y_2$ or $y_1=1-y_2$. Hence $x_1=x_2$. 

\end{proof}

%%%%%%%%%%%%%%%%%%%%%%%%%%%%%%%%%%
%   section 3
%%%%%%%%%%%%%%%%%%%%%%%%%%%%%%%%%%

\section{Miscellaneous Results}

Note that $F_n(1,0)=1$ for $n\geq 1$. Also, we have the following recursion relation for $F_n(1,1)$.
$$ F_0(1,1)=0, \,\, F_1(1,1)=1,$$
$$F_n(1,1)=F_{n-1}(1,1)-F_{n-2}(1,1), \,\textnormal{for}\,\, n\geq 2.$$

It follows that 
 $$F_2(1,1)=1,  F_3(1,1)=0,  F_4(1,1)=-1,  F_5(1,1)=-1,  F_6(1,1)=0.$$
Then we have
$$
F_n(1,1) = \left\{
        \begin{array}{ll}
            0 & ,\quad n \equiv 0, 3 \pmod{6},\\[0.3cm]
            1 & ,\quad n \equiv 1, 2 \pmod{6}, \\[0.3cm]
           -1 & ,\quad n \equiv 4, 5 \pmod{6}.
        \end{array}
    \right.
$$

\begin{thm}
Assume that $F_n(1,x)$ is a PP of $\f_q$. If $p=2$, then $3|n$. If $p$ is an odd prime, then $n\not\equiv 1, 2 \pmod{6}$. 
\end{thm}

\begin{proof}
If $p=2$, since $F_n(1,x)$ is a PP of $\f_q$ and $F_n(1,0)=1$, clearly $3|n$. If $p$ is an odd prime, then a similar argument shows that $n\not\equiv 1, 2 \pmod{6}$.
\end{proof}

Let $p$ be odd. We show that the $n$-th reversed Dickson polynomial of the third kind $F_n(1,x)$ can be written explicitly. For $n\geq 0$, define
$$f_n(x)=\displaystyle\sum_{j\geq 0} \,\,\binom{n}{2j+1}\,\,x^j.$$

\begin{prop}\label{P3.2}
Let $p$ be an odd prime. Then in $\f_q[x]$, 
$$F_n(1,x)=\Big(\frac{1}{2}\Big)^{n-1}\,f_n(1-4x).$$
In particular, $F_n(1,x)$ is a PP of $\f_q$ if and only if $f_n(x)$ is a PP of $\f_q$. 
\end{prop}

\begin{proof}
Let $x\in \f_q$. There exists $y\in \f_{q^2}$ such that $x=y(1-y)$. 
If $x\neq \frac{1}{4}$, we have 
$$F_n(1,x)=\displaystyle\frac{y^n-(1-y)^n}{2y-1}.$$
Let $u=2y-1$. Then we have
\[
\begin{split}
F_n(1,x)&=\frac{1}{u}\Big\{\Big(\frac{1+u}{2}\Big)^n-\Big(\frac{1-u}{2}\Big)^n\Big\}\cr
&=\Big(\frac{1}{2}\Big)^n\,\, \frac{1}{u}\,\, \Big\{(1+u)^n-(1-u)^n\Big\}\cr
&=\Big(\frac{1}{2}\Big)^{n-1}\,\,\displaystyle\sum_{j\geq 0}\,\binom{n}{2j+1}\,\,u^{2j}.
\end{split}
\]
Then we have
$$F_n(1,x)=\Big(\frac{1}{2}\Big)^{n-1}\,\,f_n(u^2).$$
Since $u=2y-1$, $u^2=1-4y(y-1)$.
\[
\begin{split}
F_n(1,x)&=\Big(\frac{1}{2}\Big)^{n-1}\,\,f_n(1-4y(y-1))\cr
&=\Big(\frac{1}{2}\Big)^{n-1}\,\,f_n(1-4x).
\end{split}
\]

If $x=\frac{1}{4}$, since $f_n(0)=n$, we have
$$F_n(1,x)=\frac{n}{2^{n-1}}=\Big(\frac{1}{2}\Big)^{n-1}\,f_n(0)=\Big(\frac{1}{2}\Big)^{n-1}\,f_n(1-4x).$$

Clearly, $F_n(1,x)$ is a PP of $\f_q$ if and only if $f_n(x)$ is  PP of $\f_q$. 
\end{proof}

\begin{thm}
Let $p$ be an odd prime, $q$ a power of $p$, and $n$ be a nonnegative even integer with $p\nmid \,n$. If $F_n(1,x)$ is a PP of $\f_q$, then $n\equiv 0 \pmod{4}$ and $(\lfloor \frac{n-1}{2} \rfloor, q-1)=1$. 
\end{thm}

\begin{proof}
Assume that $F_n(1,x)$ is a PP of $\f_q$. Then by Proposition~\ref{P3.2}, $f_n(x)$ is  PP of $\f_q$. 

Let $x_0 \in \f_q$ such that $f_n(x_0)=0$. $f_n(0)=n\neq 0$. Since $f_n$ is a PP of $\f_q$, $x_0 \neq 0$. 
\[
\begin{split}
f_n(x_0)&=\displaystyle\sum_{j\geq 0} \,\,\binom{n}{2j+1}\,\,x_0^j
\end{split}
\]

\[
\begin{split}
f_n(x_0^{-1})&=\displaystyle\sum_{j\geq 0} \,\,\binom{n}{2j+1}\,\,x_0^{-j}
\end{split}
\]

It is easy to see that 
\[
\begin{split}
f_n(x_0)&=x_0^{\lfloor \frac{n-1}{2} \rfloor}f_n(x_0^{-1}).
\end{split}
\]
So $f_n$ is a self-reciprocal. Since $f_n(x_0)=0$ and $x_0^{\lfloor \frac{n-1}{2} \rfloor}\neq 0$,  $f_n(x_0^{-1})=0$. Since $f_n$ is a PP of $\f_q$, $x_0=x_0^{-1}$, i.e. $x_0=\pm 1$. 
$$f_n(1)=\displaystyle\sum_{j\geq 0} \,\,\binom{n}{2j+1}\,=\,2^{n-1}\,\neq 0.$$
Therefore, $x_0=-1$. 
\[
\begin{split}
0&=f_n(-1)\cr
&=\displaystyle\sum_{j\geq 0} \,\,\binom{n}{2j+1}\,\,(-1)^{j}\cr
&=\displaystyle\sum_{j\equiv 1 \pmod{4}}\,\,\binom{n}{j} - \displaystyle\sum_{j\equiv 3 \pmod{4}} \,\,\binom{n}{j} \cr
&=\frac{1}{4}\Big[2^{n+1}+i^{-1}(1+i)^n-i^{-1}(1-i)^n \Big]\cr
&-\frac{1}{4}\Big[2^{n+1}-i^{-1}(1+i)^n+i^{-1}(1-i)^n \Big]\text{(by \cite[Eq.~5.5]{Hou-2007})}\cr
&=\frac{1}{2i}\Big[(1+i)^n-(1-i)^n \Big] \cr
&=\frac{i}{2}\Big[(1-i)^n-(1+i)^n \Big] \cr
&=\frac{i}{2}\Big[(\sqrt{2}\,\,e^{-\frac{\pi}{4} i})^n- (\sqrt{2}\,\,e^{\frac{\pi}{4} i})^n\Big] \cr
&=2^{\frac{n}{2}-1}\,\,i\,\,\Big[e^{-\frac{n\pi}{4} i}- e^{\frac{n\pi}{4} i}\Big].
\end{split}
\]

We have $\Big[e^{-\frac{n\pi}{4} i}- e^{\frac{n\pi}{4} i}\Big]=0.$ It follows that $n\equiv 0 \pmod{4}$.

Let $(\lfloor \frac{n-1}{2} \rfloor, q-1)=d>1.$ Let $\epsilon \in \f_q^*$ such that $o(\epsilon)=d$. Then 
\[
\begin{split}
f_n(\epsilon)&=\epsilon^{\lfloor \frac{n-1}{2} \rfloor}f_n(\epsilon^{-1}).
\end{split}
\]
\[
\begin{split}
f_n(\epsilon)&=f_n(\epsilon^{-1}).
\end{split}
\]
But $\epsilon\neq \epsilon^{-1}$. This contradicts the fact  that $f_n$ is a PP of $\f_q$. Hence  $(\lfloor \frac{n-1}{2} \rfloor, q-1)=1.$ 
\end{proof}

\begin{lem}\label{L3.4}(See \cite{Hou-Ly-FFA-2010})
Let $\epsilon \neq 0,1$ in some extension of $\f_q$ ($q$ odd) and let $y=\displaystyle\frac{\epsilon +1}{\epsilon -1}$. Then $y^2\in \f_q$ if and only if $\epsilon^{q+1}=1$ or $\epsilon^{q-1}=1$. 
\end{lem}

\begin{thm}
Let $p>3$ be an odd prime and $n\geq 0$ be an integer with $3 | n$. If $F_n(1,x)$ is a PP of $\f_q$, then $(n,q^2-1)=3$. 
\end{thm}

\begin{proof}
Since $p>3$, we have $q\equiv 1\,\,\textnormal{or}\,\,-1\,\,\pmod{3}$. Since $3|n$, we have $3|(n, q^2-1)$. We show that $(n,q^2-1)\leq 3$. Assume to the contrary that $(n,q^2-1)>3$. Let 
$$E=\{ \epsilon \in \f_{q^2}^{*}:\,\epsilon \neq 1,\, \epsilon^{(n,q+1)}=1\,\textnormal{or}\,\,\epsilon^{(n,q-1)}=1\}.$$
\[
\begin{split}
|E|&=|\{ \epsilon \in \f_{q^2}^{*}:\,\epsilon \neq 1,\, \epsilon^{(n,q+1)}=1\}| +|\{ \epsilon \in \f_{q^2}^{*}:\,\epsilon \neq 1,\, \epsilon^{(n,q-1)}=1\}| \cr
&-|\{ \epsilon \in \f_{q^2}^{*}:\,\epsilon \neq 1,\, \epsilon^{(n,q-1,q+1)}=1\}| \cr
&=((n,q+1)-1) + ((n,q-1)-1) - 0\cr
&=(n,q+1)+ (n,q-1) -2.
\end{split}
\]
Since $(n,q+1)(n,q-1)=(n,q^2-1)\geq 6$, we have $|E|\geq $4. Let $\epsilon_1, \epsilon_2, \epsilon_3 \in E$ be distinct and let $y_i=\frac{\epsilon_i -1}{\epsilon_i +1},\,\,i=1,2,3$. By Lemma~\ref{L3.4}, $y_i\in \f_q$. Since $\epsilon_i=\frac{1+y_i}{1-y_i}$, we have $\Big(\frac{1+y_i}{1-y_i}\Big)^n=1$, i.e. $(1+y_i)^n=(1-y_i)^n$, i.e.
$$f_n(y_i^2)= \frac{1}{2y_i}\,\, \{(1+y_i)^n-(1-y_i)^n\}=0.$$ 

Since $\epsilon_1, \epsilon_2,\,\textnormal{and}\,\epsilon_3$ are distinct, $y_1, y_2, \,\textnormal{and}\,y_3$ are distinct. This contradicts the fact that $f_n$ is a PP of $\f_q$. 

\end{proof}

%%%%%%%%%%%%%%%%%%%%%%%%%%%%%%%%%%
%   section 4
%%%%%%%%%%%%%%%%%%%%%%%%%%%%%%%%%%

\section{Computation of $\sum_{a\in \f_q}F_{n}(1,a)$. }

We compute the sum $\sum_{a\in \f_q}F_{n}(1,a)$ in this section. The result provides a necessary condition for $F_n(1,x)$ to be a PP of $\f_q$. By Theorem~\ref{T2.6}, we have
\begin{equation}\label{E4.1}
\begin{split}
\displaystyle\sum_{n=0}^{\infty}\,F_{n}(1,x)\,z^n&=\displaystyle\frac{z}{1-z+xz^2}\cr
&=\displaystyle\frac{z}{1-z}\,\,\displaystyle\frac{1}{1-(\frac{z^2}{z-1})\,x}\cr
&=\displaystyle\frac{z}{1-z}\,\,\displaystyle\sum_{k\geq 0} \Big(\frac{z^2}{z-1}\Big)^k\,\,x^k\cr
&=\displaystyle\frac{z}{1-z}\,\,\Big[1+\displaystyle\sum_{k=1}^{q-1}\displaystyle\sum_{l\geq 0} \Big(\frac{z^2}{z-1}\Big)^{k+l(q-1)}\,\,x^{k+l(q-1)} \Big]\cr
&\equiv \displaystyle\frac{z}{1-z}\,\,\Big[1+\displaystyle\sum_{k=1}^{q-1}\displaystyle\sum_{l\geq 0} \Big(\frac{z^2}{z-1}\Big)^{k+l(q-1)}\,\,x^k \Big]\,\,\,\pmod{x^q-x}\cr
&=\displaystyle\frac{z}{1-z}\,\,\Big[1+\displaystyle\sum_{k=1}^{q-1}\displaystyle\frac{(\frac{z^2}{z-1})^{k}}{1-(\frac{z^2}{z-1})^{q-1}}\,\,x^k \Big]\cr
&=\displaystyle\frac{z}{1-z}\,\,\Big[1+\displaystyle\sum_{k=1}^{q-1}\displaystyle\frac{(z-1)^{q-1-k}\,\,z^{2k}}{(z-1)^{q-1} - z^{2(q-1)}}\,\,x^k \Big]\cr
\end{split}
\end{equation}

Since $F_{n_1}(1,x)=F_{n_2}(1,x)$ for any $x\in \f_{q}\setminus \{\frac{1}{4}\}$ when $n_1, n_2 >0$ and $n_1\equiv n_2 \pmod{q^2-1}$, we have the following for all $x\in \f_{q}\setminus \{\frac{1}{4}\}$. 
\begin{equation}\label{E4.2}
\begin{split}
\displaystyle\sum_{n\geq 0} \,F_{n}\,z^n&= \displaystyle\sum_{n\geq 1} \,F_{n}\,z^n\cr
&=\displaystyle\sum_{n=1}^{q^2-1}\,\, \displaystyle\sum_{l\geq 0} \,F_{n+l(q^2-1)}\,z^{n+l(q^2-1)}\cr
&= \displaystyle\sum_{n=1}^{q^2-1}\,\,F_n\,\, \displaystyle\sum_{l\geq 0}\,z^{n+l(q^2-1)}\cr
&=\displaystyle\frac{1}{1-z^{q^2-1}} \displaystyle\sum_{n=1}^{q^2-1}\,F_n\,z^n
\end{split}
\end{equation}

Combining \eqref{E4.1} and \eqref{E4.2} gives

\[
\begin{split}
\displaystyle\frac{1}{1-z^{q^2-1}} \displaystyle\sum_{n=1}^{q^2-1}\,F_n\,z^n = \displaystyle\frac{z}{1-z}\,\,\Big[1+\displaystyle\sum_{k=1}^{q-1}\displaystyle\frac{(z-1)^{q-1-k}\,\,z^{2k}}{(z-1)^{q-1} - z^{2(q-1)}}\,\,x^k \Big],\,\,\textnormal{for all}\,\,x\in \f_{q}\setminus \{\tfrac{1}{4}\},
\end{split}
\]

i.e.

\[
\begin{split}
\displaystyle\sum_{n=1}^{q^2-1}\,F_n\,z^n = \displaystyle\frac{z\,(z^{q^2-1}-1)}{z-1} +\,\,h(z)\,\,\displaystyle\sum_{k=1}^{q-1}(z-1)^{q-1-k}\,\,z^{2k}\,\,x^k, \,\,\textnormal{for all}\,\,x\in \f_{q}\setminus \{\tfrac{1}{4}\},
\end{split}
\]

where
$$h(z)=\displaystyle\frac{z\,(z^{q^2-1}-1)}{(z-1)\,[(z-1)^{q-1}-z^{2(q-1)}]}.$$

Note that 

\[
\begin{split}
h(z) &=\displaystyle\frac{z\,(z^{q^2-1}-1)}{(z-1)^{q}-z^{2(q-1)}(z-1)}\cr
&= \displaystyle\frac{z\,(z^{q^2-1}-1)}{(1-z^{q-1})\,(z^q-z^{q-1}-1)}\cr
&=  \displaystyle\frac{z\,(z^{q^2}-z)}{(z-z^{q})\,(z^q-z^{q-1}-1)}\cr
&= \displaystyle\frac{z\,(-1-(z-z^q)^{q-1})}{z^q-z^{q-1}-1}\cr
\end{split}
\]

Let $\displaystyle\sum_{k=1}^{q^2-q+1}\,b_kz^k = z \,\,(-1-(z-z^q)^{q-1})$. 

Write $k=\alpha + \beta q$ where $0\leq \alpha, \beta \leq q-1$. Then we have the following.

$$
b_k = \left\{
        \begin{array}{ll}
           (-1)^{\beta +1} \,\binom{q-1}{\beta} &  \textnormal{if}\,\,\alpha +\beta =q,\\[0.3cm]
            -1 &  \textnormal{if}\,\,\alpha +\beta =1, \\[0.3cm]
            0 &  \textnormal{otherwise}.
        \end{array}
    \right.
$$

A computation similar to (4.4) in \cite{Hong-Qin-Zhao-FFA-2016} yields 

\[
\begin{split}
&\displaystyle\sum_{n=1}^{q^2-1}\, \Big(\displaystyle\sum_{a\in \f_q}\,F_n(1, a)\Big) z^n \cr &=\displaystyle\sum_{n=1}^{q^2-1}\,F_n(1, \frac{1}{4})\, z^n - \displaystyle\frac{z(1-z^{q^2-1})}{1-z} - h(z)\,z^{2(q-1)}-h(z)\displaystyle\sum_{j=1}^{q-1}\,(z-1)^{q-1-j}\,z^{2j}\,\Big(\frac{1}{4}\Big)^j,
\end{split}
\]

which implies

\begin{equation}\label{E4.3}
\begin{split}
&\displaystyle\sum_{n=1}^{q^2-1}\, \Big(\displaystyle\sum_{a\in \f_q}\,F_n(1, a)\Big) z^n \cr &=\displaystyle\sum_{n=1}^{q^2-1}\, \frac{n}{2^{n-1}}\, z^n - \displaystyle\frac{z(1-z^{q^2-1})}{1-z} - h(z)\,z^{2(q-1)}-h(z)\displaystyle\sum_{j=1}^{q-1}\,(z-1)^{q-1-j}\,z^{2j}\,\Big(\frac{1}{4}\Big)^j,
\end{split}
\end{equation}

where 

$$h(z)= \displaystyle\frac{1}{z^q-z^{q-1}-1}\,\,\displaystyle\sum_{k=1}^{q^2-q+1}\,b_kz^k.$$

For all integers $1\leq n\leq q^2-1$, define
$$f_n:=\displaystyle\sum_{a\in \f_q} F_n(1,a).$$

Then from \eqref{E4.3}, we have 

\begin{equation}\label{E4.4}
\begin{split}
&(z^q-z^{q-1}-1)\displaystyle\sum_{n=1}^{q^2-1}\, \Big(f_n - \frac{n}{2^{n-1}}\Big) z^n \cr &=(1+z^{q-1}-z^q)\,\displaystyle\sum_{k=1}^{q^2-1}\,z^k\, - \Big(\,z^{2(q-1)}+ \,\displaystyle\sum_{j=1}^{q-1}\,(z-1)^{q-1-j}\,z^{2j}\,\Big(\frac{1}{4}\Big)^j \Big) \,\Big( \displaystyle\sum_{k=1}^{q^2-q+1}\,b_kz^k \Big).
\end{split}
\end{equation}

Let $d_n=f_n - \displaystyle\frac{n}{2^{n-1}}$ and the right hand side of \eqref{E4.4} be $\displaystyle\sum_{k=1}^{q^2+q-1}\,c_kz^k$.

Then we have

\begin{equation}\label{E4.5}
(z^q-z^{q-1}-1)\,\displaystyle\sum_{n=1}^{q^2-1}\, d_n z^n = \displaystyle\sum_{k=1}^{q^2+q-1}\,c_kz^k.
\end{equation}

\begin{prop}\label{P4.1} (See \cite{Hong-Qin-Zhao-FFA-2016}) By comparing the coefficient of $z^i$ on both sides of \eqref{E4.5}, we have the following.\\

\noindent $d_j=-c_j$ if $1\leq j\leq q-1$;\\
$d_q=c_1-c_q$;\\
$d_{lq+j}=d_{(l-1)q+j} - d_{(l-1)q+j+1} - c_{lq+j}$ if $1\leq l\leq q-2$ and $1\leq j\leq q-1$;\\
$d_{lq}=d_{(l-1)q} - d_{(l-1)q+1} - c_{lq}$ if $2\leq l\leq q-2$;\\
$d_{q^2-q+j}=\displaystyle\sum_{i=j}^{q-1}\,c_{q^2+i}$ if $0\leq j\leq q-1$.
\end{prop}

The following theorem is an immediate consequence of Proposition~\ref{P4.1} and the fact that $d_n:=\displaystyle\sum_{a\in \f_q} F_n(1,a)  - \frac{n}{2^{n-1}}$.

\begin{thm}
Let $c_k$ be defined as in \eqref{E4.5} for $1\leq k\leq q^2+q-1$. Then we have the following. \\

$\displaystyle\sum_{a\in \f_q} F_j(1,a)= -c_j + \frac{j}{2^{j-1}}\,\, \textnormal{if}\,\, 1\leq j\leq q-1;$\\

$\displaystyle\sum_{a\in \f_q} F_q(1,a)= c_1-c_q;$

$\displaystyle\sum_{a\in \f_q} F_{lq+j}=\displaystyle\sum_{a\in \f_q} F_{(l-1)q+j} - \displaystyle\sum_{a\in \f_q} F_{(l-1)q+j+1} - c_{lq+j}+\displaystyle\frac{2^q(1-j)+2j}{2^{lq+j}}$ if $1\leq l\leq q-2$ and $1\leq j\leq q-1$;\\

$\displaystyle\sum_{a\in \f_q} F_{lq}=\displaystyle\sum_{a\in \f_q} F_{(l-1)q} - \displaystyle\sum_{a\in \f_q} F_{(l-1)q+1} - c_{lq}+\displaystyle\frac{1}{2^{(l-1)q}}$ if $2\leq l\leq q-2$;\\

$\displaystyle\sum_{a\in \f_q} F_{q^2-q+j}=\displaystyle\sum_{i=j}^{q-1}\,c_{q^2+i}+\frac{j}{2^{q^2-q+j-1}}$ if $0\leq j\leq q-1$.

\end{thm}

%%%%%%%%%%%%%%%%%%%%%%%%%%%%%%%%%%
%   section 5
%%%%%%%%%%%%%%%%%%%%%%%%%%%%%%%%%%

\section{Acknowledgements}

The author would like to thank G. C. Greubel for pointing out an important connetion of the reversed Dickson polynomials of the third kind to Jacobsthal polynomials and integer sequences. He pointed out the following. 

\begin{enumerate}

\item Proposition 2.5 is related to the Jacobsthal polynomials by $F_{n}(1,x) = J_{n}(-x/2)$. 

\item The generating function presented in Theorem 2.8 is another indicator to the connection to Jacobsthal polynomials. 

\item The third indicator is the reduction presented near the end of Theorem 2.10, namely, $F_{n}(1,\frac{1}{4}) = \displaystyle\frac{n}{2^{n-1}}$.

\item The number set $\{0,1,1,0,-1,-1,0,1,1,..\}$ presented in miscellaneous results, directly before Theorem 3.1, is given as sequence $A010892$ in the On-line Encyclopedia of Integer Sequences. The starting point of this sequence mentioned is offset by one index. An alternate sequence is $A128834$.  

\end{enumerate}

%%%%%%%%%%%%%%%%%%%%%%%%%%%%%%%%%%%%%%

\end{document}